\documentclass[10pt]{amsart}

\usepackage{amsmath}
\usepackage{amsthm}
\usepackage{amsfonts}
\usepackage{amssymb,latexsym}
\usepackage[all]{xy}
\usepackage{graphicx}
\usepackage[mathscr]{eucal}
\usepackage{verbatim}
\usepackage{hyperref}

\theoremstyle{plain}
    \newtheorem{thm}{Theorem}[section]
    \newtheorem{prop}[thm]{Proposition}

    \newtheorem{cor}[thm]{Corollary}

\theoremstyle{definition}
    \newtheorem{defn}[thm]{Definition}

\theoremstyle{remark}
    
    \newtheorem{example}[thm]{Example}

\numberwithin{equation}{section}

\newcommand{\rar}{\ensuremath{\rightarrow}}
\newcommand{\lrar}{\ensuremath{\longrightarrow}}

\newcommand{\la}{\langle}
\newcommand{\ra}{\rangle}
\newcommand{\Hom}{\textup{Hom}}
\newcommand{\StMod}{\textup{StMod}}
\newcommand{\stmod}{\textup{stmod}}
\newcommand{\Mod}{\textup{Mod}}

\newcommand{\uHom}{\underline{\Hom}}
\newcommand{\Ker}{\textup{Ker}}

\newcommand{\G}{\mathcal G}

\newcommand{\stk}[1]{\stackrel{#1}{\rightarrow}}
\newcommand{\lstk}[1]{\stackrel{#1}{\longrightarrow}}
\newcommand{\Ext}{\textup{Ext}}
\newcommand{\Gal}{\textup{Gal}}

\newcommand{\thick}{\textup{thick}}

\newcommand{\End}{\text{End}}

\newcommand{\ftwo}{\mathbb{F}_2}

\def\Tate{\operatorname{\hat{H}}\nolimits}
\def\HHHH{\operatorname{\hat{H}}\nolimits}
\def\Hom{\operatorname{Hom}\nolimits}
\def\Ext{\operatorname{Ext}\nolimits}
\def\hExt{\operatorname{\widehat{Ext}}\nolimits}

\begin{document}

\title[Auslander-Reiten sequences]{Auslander-Reiten sequences as appetizers for homotopists and arithmeticians}
\date{\today}

\author{Sunil K. Chebolu}
\address{Department of Mathematics \\
Illinois State University \\
Normal, IL 61790, USA} \email{schebol@ilstu.edu}

\author{J\'{a}n Min\'{a}\v{c} $^{*}$}
\address{Department of Mathematics\\
University of Western Ontario\\
London, ON N6A 5B7, Canada}
\email{minac@uwo.ca}
\thanks{J\'{a}n Min\'{a}\v{c} was supported in part by Natural Sciences and Engineering Research Council of Canada grant R3276A01}


\begin{abstract}
We introduce Auslander-Reiten sequences for group algebras and give
several recent applications. The first part of the paper is devoted to some
fundamental problems in Tate cohomology  which are motivated by
homotopy theory. In the second part of the paper we interpret Auslander-Reiten
sequences in the context of Galois theory and connect them to some important
arithmetic objects.
\end{abstract}

\dedicatory{Dedicated to Professor John Labute with admiration, respect and friendship.}

\maketitle
\thispagestyle{empty}

\tableofcontents

\section{Introduction}

In the past the second author had an opportunity to discuss with John Labute
some beautiful aspects of modular representation theory and they both came to
the conclusion that the techniques from modular representation theory could be
applied effectively to study problems in number theory, particularly in field
theory and Galois theory. But unfortunately some of these techniques from
representation theory are not well-known to people working in number theory.

Several times recently, we ran into Auslander-Reiten sequences in our work and
we began to appreciate the power and elegance of these sequences. These
sequences are some special non-split short exact sequences which are very
close to being split sequences.  They were introduced by Auslander and Reiten
in the early 70s \cite{AR}, and they have been among the most important tools
from modular representation theory.

The goal of this article is to introduce Auslander-Reiten sequences and to
give some applications. This article is by no means  a historical survey of
Auslander-Reiten theory; it focuses only on some selective applications from
our recent joint work. The treatment here will be particularly suitable for
homotopy theorists and number theorists.  We hope that the readers who are not
specialists in representation theory might be inspired further to learn more
about Auslander-Reiten sequences and investigate possible Galois field
theoretic applications.

We would like to thank the referee for his/her valuable suggestions which have
helped us to improve the exposition of our paper

\section{Auslander-Reiten sequences} Throughout we let $G$ denote a finite group
and $k$ a field of characteristic $p$ that divides the order of $G$. We denote
by $\Lambda$ the group algebra $kG$. Although Auslander-Reiten theory is often
studied in the broader context of Artin algebras, for simplicity we restrict
ourselves to the most interesting case of group algebras.

\begin{defn}
A short exact sequence of $\Lambda$-modules
\[ \epsilon  \colon  \quad 0 \lrar M \lstk{\alpha} E \lstk{\beta} N \lrar 0\]
is an Auslander-Reiten (AR) sequence if the following three conditions are
satisfied:
\begin{itemize}
\item $\epsilon$ is a non-split sequence.
\item $\alpha$ is left almost split, i.e., any map $\phi \colon M \rar M'$ that is not a split monomorphism
factors through $E$ via $\alpha$.
\item $\beta$ is right almost split, i.e., any map $\psi \colon N' \rar N$ that is not a split epimorphism
factors through $E$ via $\beta$.
\end{itemize}
\[
\xymatrix{
&    & & N'\ar[d]^{\psi} \ar@{.>}[dl] &\\
0\ar[r] & M \ar[r]^{\alpha} \ar[d]_{\phi} & E \ar[r]^{\beta} \ar@{.>}[dl] & N \ar[r] &  0\\
& M' & & &
}
\]
\end{defn}

There are several other equivalent definitions of an Auslander-Reiten
sequence. The one we have given here is perhaps the most symmetric of all.
Note that a split short exact sequence satisfies the last two conditions. This
justifies why Auslander-Reiten sequences are also called \textit{almost split
sequences}. Further, we note that these conditions also imply that the terms
$M$ and $N$ in an AR sequence have to be indecomposable and non-projective.

\begin{example} \textbf{(One-up and One-down)} Let us start with a very simple example.
Let $\Lambda = k[x]/x^p$ where $p$ is the characteristic of $k$. Observe
that $\Lambda$ is a group algebra $kG$ with a $G$ cyclic group
of order $p$. Then for each integer
$i$ strictly between $0$ and $p$,  the short exact sequence
\[ 0 \lrar k[x]/x^{i} \lstk{(j, -\pi)} k[x]/x^{i+1} \oplus k[x]/x^{i-1}  \lstk{}  k[x]/x^{i} \lrar 0\]
is easily shown (exercise left to the reader)
to be an Auslander-Reiten sequence. Here $j$ is the obvious injection into the first summand of the
middle term, and $\pi$ is the obvious surjection on the second summand.
\end{example}

The remarkable result is the existence of these
sequences:

\begin{thm} \cite{aus-rei-sma}
Given any indecomposable non-projective $\Lambda$-module $N$, there exists a
unique (up to isomorphism of short exact sequences) Auslander-Reiten sequence
$\epsilon$ ending in $N$. Moreover the first term $M$ of the AR-sequence
$\epsilon$ is isomorphic to $\Omega^2 N$.
\end{thm}

In Section \ref{sec:Krause} we present a non-standard proof due to Krause of
the existence of AR sequences. This uses the Brown representability theorem
from algebraic topology. For now, we assume existence and show uniqueness.

Suppose we have two AR-sequences ending in an indecomposable non-projective $kG$-module $L$:
\[ 0 \lrar M \lstk{\alpha} E \lstk{\beta} L \lrar 0\]
\[ 0 \lrar M' \lstk{\alpha'} E' \lstk{\beta'} L \lrar 0\]
The defining property of AR sequences then gives the following commutative diagram.
\[
\xymatrix{
0 \ar[r] & M  \ar[r]^{\alpha} \ar@{.>}[d]_{f}  & E \ar[r]^{\beta} \ar@{.>}[d] & L \ar[r] \ar[d]^{=} & 0 \\
0 \ar[r] & M' \ar[r]^{\alpha'} \ar@{.>}[d]_{g}  & E' \ar[r]^{\beta'} \ar@{.>}[d] & L \ar[r] \ar[d]^{=} & 0 \\
0 \ar[r] & M \ar[r]^{\alpha} & E \ar[r]^{\beta} & L \ar[r] & 0
}
\]
The composite $\phi := g f$  cannot be nilpotent, otherwise $\beta$ would be
a split epimorphism. (This is an easy diagram-chase exercise.) Now since $M$ is a module of
finite length, for sufficiently large values of $n$, we must have
\[ M \cong \text{Im}(\phi^n) \oplus \Ker(\phi^n).  \]
Since $M$ is indecomposable, and $\phi$ is not nilpotent, it follows that $M
\cong \text{Im}(\phi^n)$ and $0 = \Ker(\phi^n)$. This shows that $\phi^n$, and
hence $\phi$, is an isomorphism.  Therefore $f$ is the left inverse of $g$.
Interchanging the roles of the two AR-sequences above, it follows similarly
that $f$ is the right inverse of $g$. Thus we have $M \cong M'$. The five
lemma now tells us that $E \cong E'$. Thus the two AR-sequences are
isomorphic as short exact sequences.

\section{Stable module category}
We now introduce the stable module category. We will work mostly in this
category for the first half of the paper. The \emph{stable module category}
$\StMod(kG)$ of $G$ is the category obtained from the category $\Mod(kG)$ of
(left) $kG$-modules by killing the projective modules. More precisely, it is
the category whose objects are $kG$-modules and whose morphisms are
equivalence classes of $kG$-module homomorphisms, where two homomorphisms are
equivalent if their difference factors through a projective module (such a map
is called \emph{projective}). If $M$ and $N$ are two $kG$-modules, then we
write $\uHom_{kG}(M, N)$ for the $k$-vector space of maps between $M$ and $N$
in the stable module category. It is easy to see that a $kG$-module is
projective if and only if it is the zero object in $\StMod(kG)$. The compact
stable category $\stmod(kG)$ is the full subcategory of $\StMod(kG)$
consisting of the
 finitely generated (equivalently, finite-dimensional) $kG$-modules.

The category $\StMod(kG)$ admits a natural tensor triangulated structure which
we now explain.  If $M$ is a $kG$-module then define its desuspension
$\Omega\, M$ to be the kernel of a surjective map from a projective module to $M$. By
Schanuel's Lemma, the new module thus obtained is well-defined up to
projective summands and therefore defines an endofunctor
\[ \Omega\colon \StMod(kG) \lrar \StMod(kG) . \]
Using the fact that $kG$ is a Frobenius algebra (injectives and projective
coincide) one can show that $\Omega$ is an equivalence. The inverse functor
$\Omega^{-1}$ sends $M$ to the cokernel of an injective map from $M$ to an
injective module. This allows one to talk about negative powers of $\Omega$.

The importance of the stable module category can be seen in the following
useful and well-known fact.

\begin{prop}
Let $M$ and $N$ be $kG$-modules. If $n$ is any positive integer, then
\[
  \uHom_{kG}(\Omega^n \, M, N) \cong \Ext_{kG}^{\, n}\,(M, N) \cong
  \uHom_{kG}(M, \Omega^{-n}\,N).
\]
\end{prop}

In particular, group cohomology is elegantly encoded in the stable module
category. In fact, allowing $n$ to be any integer, one recovers the Tate
cohomology ring of $G$ in this way as the (graded) ring of self maps of the
trivial representation. Under this isomorphism, the multiplication in the Tate
ring of $G$ corresponds to composition in $\uHom(\Omega^*\,k, k)$.

Now we describe the triangles in the stable module category. These arise from
short exact sequences in the module category in the following way. Suppose
\[
  0 \lrar L \lrar M \lrar N \lrar 0
\]
is a short exact sequence of $kG$-modules. This short exact sequence
determines a unique class in $\Ext^1_{kG}(N, L)$ which under the isomorphism
given in the above proposition corresponds to a map $N \rar \Omega^{-1} L$.
Thus we get a diagram
\[
  L \lrar M \lrar N \lrar \Omega^{-1}L. \ \ \ (*)
\]
A triangle in $\StMod(kG)$ is any such diagram which is isomorphic to a
diagram which arises as above (*). These triangles and the suspension functor
satisfy the standard axioms for a triangulated category; see
\cite{carlson-modulesandgroupalgebras} for example.

The tensor product of two $kG$-modules (with diagonal $G$-action on the tensor
product of $k$-vector spaces) descends to a well-behaved tensor product in the
stable module category. (This follows from the fact that if $M$ and $N$ are
$kG$-modules with either $M$ or $N$ projective, then $M \otimes_k N$ is
projective.) The trivial representation $k$ of $G$ serves as a unit for the
tensor product.

\section{Krause's proof of the existence of AR sequences}\label{sec:Krause}
The proof of the existence of an Auslander-Reiten sequence can be found in \cite{aus-rei-sma}.
We present a different avatar of Auslander-Reiten sequences in the stable module category,
and prove their existence following Krause \cite{Krause-AR}.

\begin{defn}
A triangle $A \lstk{\alpha} B \lstk{\beta} C \lstk{\gamma} \Omega^{-1} A $ in
$\StMod(kG)$ is an Auslander-Reiten triangle if the following three conditions
hold.
\begin{itemize}
\item $\gamma \ne 0$.
\item Every map $f \colon A \lrar A'$ which is not a section factors through $\alpha$.
\item Every map $g \colon C' \lrar C$ which is not a retraction factors through $\beta$.
\end{itemize}
\end{defn}

The following theorem is crucial in what follows. This is the Brown
representability theorem for $\StMod(kG)$.

\begin{thm} (Brown)
A contravariant functor $F \colon \StMod(kG) \lrar \text{Ab}$ that is exact and sends coproducts to products
is representable, i.e., there exists a module $M_F$ in $\StMod(kG)$ such that
\[ F(-) \cong \uHom_{kG}(-, M_F).\]
\end{thm}

Now fix a finitely generated non-projective $kG$-module $Z$ that has a local
endomorphism ring, and set $\Gamma = \uHom_{kG} (Z, Z)$. Let $I$ be the
injective envelope of $\Gamma/ \text{rad }  \Gamma$. Consider the functor
\[ \Hom_{\Gamma} (\uHom_{kG} (Z, -), I).\]
It is not hard to see that this is a contravariant functor that is exact and sends coproducts to products.
So by the Brown representability theorem this functor is representable:
\begin{equation} \label{eq:brown}
 \Psi \colon \Hom_{\Gamma} (\uHom_{kG} (Z, -), I) \cong \uHom_{kG} (- , T_Z
I).
\end{equation}

Finally consider the composite $\omega \colon \Gamma \lstk{\pi}
\Gamma/\text{rad} \; \Gamma \hookrightarrow I$, and denote by  $\gamma$ the
map $\Psi (\omega)$. Note that $\gamma$ is a map from $Z \lrar T_Z I$.
Extending this map to a triangle in $\StMod(kG)$, we get
\[\Omega T_Z I \lrar Y \lrar Z \lstk{\gamma} T_Z I.\]

\begin{thm} [\cite{Krause-AR}]
The triangle $\Omega T_Z I \lstk{\alpha} Y \lstk{\beta} Z \lstk{\gamma} T_Z I $ is an
Auslander-Reiten triangle.
\end{thm}

\begin{proof}

$\gamma \ne 0$: Recall that $\gamma = \Psi(\omega)$, where $\Psi$ is an isomorphism of functors, and
$\omega$ is a non-zero map. Therefore $\gamma \ne 0$.

$\beta$ has the right almost split property: That is, given a map $\rho \colon
Z' \lrar Z$ that is not a retraction, we have to show that there exists a map
$ \rho' \colon Z'\lrar Y$ making the following diagram commutative.
\[
\xymatrix{
& & Z' \ar[d]^{\rho} \ar@{.>}[dl]_{\rho'} & \\
\Omega T_Z I \ar[r] & Y \ar[r] & Z \ar[r]^{\gamma} & T_Z I
}
\]
For this, it is enough to know that $\gamma \rho = 0$, since the bottom row in
the above diagram is an exact triangle.
Since $\rho$ is not a retraction, $\text{Im} \, \uHom_{kG}(Z, Z') \lstk{\rho_*}
\uHom_{kG}(Z , Z) = \Gamma$ does not contain the identity map. Consequently
$\text{Im} (\rho_*)$ is contained in $\text{rad} \, \Gamma$.

The naturality of the isomorphism $\Psi$ yields the following commutative
diagram.
\[
\xymatrix{
\Hom_{\Gamma} (\uHom_{kG} (Z, Z), I) \ar[r]^{\ \ \cong} \ar[d] & \uHom_{kG} (Z , T_Z I) \ar[d] \\
\Hom_{\Gamma} (\uHom_{kG} (Z, Z'), I) \ar[r]^{\ \ \ \cong} & \uHom_{kG} (Z' , T_Z I)
}
\]
Note that $\omega$ maps to zero under the left vertical map and by the
commutativity of the diagram, we see that $\gamma$  also maps to zero under
the right vertical map. Thus  $\gamma \rho = 0$, as desired.

$\alpha$ has the left almost split property: For this, it turns out that it
suffices to show that $\uHom_{kG}(T_zI, T_Z I)$ is a local ring; see
\cite{Krause-AR} for details.
\begin{eqnarray*}
\uHom_{kG}(T_zI, T_Z I) & \cong & \Hom_{\Gamma} (\uHom_{kG} (Z, T_Z I), I)\\
                        & \cong & \Hom_{\Gamma} (\Hom_{\Gamma} (\uHom_{kG} (Z, Z), I), I )\\
                        & \cong & \Hom_{\Gamma}(\Hom_{\Gamma}(\Gamma, I) , I)\\
                        & \cong & \Hom_{\Gamma} (I, I)
\end{eqnarray*}
The injective $\Gamma$-module $I$ is indecomposable since
$\Gamma/\text{rad } \Gamma$ is simple, and therefore
$\Hom_{\Gamma}(I, I)$ is a local ring. So we are done.
\end{proof}

\section{Auslander-Reiten duality}
The isomorphism given in eq. \ref{eq:brown} is just another avatar of
Auslander-Reiten duality which we now explain.  Let $M$ and $L$ be finitely
generated $kG$-modules. Then the Auslander-Reiten duality says that there is a
non-degenerate bilinear form (functorial in both variables)
\[ \Phi(-, -) \colon \uHom_{kG}(L, \Omega M) \times \uHom_{kG}(M, L) \lrar k.\]
To explain this pairing, first recall that $\Omega M$ is defined by the short
exact sequence
\[ 0 \lrar \Omega M \lstk{j} P \lstk{q} M \lrar 0,\]
where $P$ is a minimal projective cover of $M$. Given $f \colon L \rar \Omega
M$ and $g \colon M \rar L$, we then get the following diagram.
\[
\xymatrix{ P \ar[d]_{q} \ar[rr]& & P \\
M \ar[r]^{g} & L \ar[r]^f& \Omega M \ar[u]_j }
\]
The top horizontal self map of $P$ is in the image of the norm map $ \eta \colon
\End_k P \rar \End_{kG}P$ which sends $\phi$ to $\sum_{g \in G} g\cdot \phi$.
Now we define $\Phi(f, g) = tr(\theta)$, where $\eta(\theta) = jfgq$.

To see the connection between the Auslander-Reiten sequence and the natural
isomorphism \ref{eq:brown} which was responsible for the existence of
Auslander-Reiten triangles, consider the special case of Auslander-Reiten
duality when $M = k$. This gives a pairing (natural in $L$):
\[ \uHom_{kG}(L, \Omega k) \times \uHom_{kG}(k, L) \lrar k. \]
Or equivalently, a natural isomorphism of functors
\[ \Hom_k (\uHom_{kG}(k, - ), k) \cong \uHom_{kG} (- , \Omega k)\]
This is precisely the isomorphism \ref{eq:brown} when $Z = k$. Since the
representing object is unique, we get $T_k I = \Omega k$. Thus the
Auslander-Reiten triangle ending in $k$ has the form
\[ \Omega^2 k \lrar M \lrar k \lrar \Omega k. \]

In the next three sections we give three different applications of
Auslander-Reiten sequences from our recent work.

\section{Representations of $V_4$-modules}
Consider the Klein four group $V_4 = C_2 \oplus C_2$. We use AR sequences to
prove the following intriguing result which also plays a crucial role in the
proof of the classification of the indecomposable $V_4$ representations.

\begin{thm} \label{thm:Klein} Let $M$ be a projective-free $V_4$-module. Then we have
the following.
\begin{enumerate}
  \item If $l$ is the smallest positive integer such that
  $\Omega^l(k)$ is isomorphic to a submodule of
  $M$, then $\Omega^l(k)$ is a summand of $M$.
  \item Dually, if $l$ is the smallest positive integer such that
  $\Omega^{-l}(k)$ is isomorphic to a quotient module of
  $M$, then $\Omega^{-l}(k)$ is a summand of $M$.
\end{enumerate}
\end{thm}

\begin{proof}

First note that the second part of this lemma follows by dualising the first
part; here we also use the fact that $(\Omega^l\,k)^*\cong \Omega^{-l}\,k$. So
it is enough to prove the first part.

The AR sequences in the category of $kV_4$-modules are of the form (see
\cite[Appendix, p 180]{ben-trends}):
\[0 \rar \Omega^{l+2}\,k \rar \Omega^{l+1}\, k \oplus \Omega^{l+1}\, k \rar \Omega^l\,k \rar 0   \hspace{9 mm}   l \ne -1\]
\[ 0 \rar \Omega^1\,k \rar kV_4 \oplus k \oplus k \rar \Omega^{-1}\,k \rar 0  \hspace{20 mm}\]

To prove the first part of the above lemma, let $l$ be the smallest positive
integer such that $\Omega^l\,k$ embeds in a projective-free $V_4$-module $M$.
If this embedding does not split, then by the property of an almost split
sequence, it should factor through $\Omega^{l-1}\,k \oplus \Omega^{l-1}\,k$ as
shown in the diagram below.
\[
\xymatrix{ 0 \ar[r] & \Omega^{l}\,k \ar[r] \ar@{^{(}->}[d] & \Omega^{l-1}\,k
\oplus \Omega^{l-1}\,k
\ar[r] \ar@{..>}[dl]^{f \oplus g} \ar[r]& \Omega^{l-2}\,k \ar[r] & 0 \\
& M & & & & }
\]
Now if either $f$ or $g$ is injective, that would contradict the minimality of
$l$, so they cannot be injective. So both $f$ and $g$ should factor through
$\Omega^{l-2}\,k \oplus \Omega^{l-2}\,k$ as shown in the diagrams below.
\[
\xymatrix{ 0 \ar[r] & \Omega^{l-1}\,k \ar[r] \ar[d]_f & \Omega^{l-2}\,k \oplus
\Omega^{l-2}\,k
\ar[r] \ar@{..>}[dl]^{(f_1 \oplus f_2)} \ar[r] & \Omega^{l-3}\,k \ar[r] & 0 \\
& M & & & & }
\]
\[
\xymatrix{ 0 \ar[r] & \Omega^{l-1}\,k \ar[r] \ar[d]_g & \Omega^{l-2}\,k \oplus
\Omega^{l-2}\,k
\ar[r] \ar@{..>}[dl]^{(g_1 \oplus g_2)} \ar[r] & \Omega^{l-3}\,k \ar[r] & 0 \\
& M & & & & }
\]

 Proceeding in this way we can assemble all the lifts obtained
using the almost split sequences into one diagram as shown below.
\[
\xymatrix{
\Omega^l\, k \ar@{^{(}->}[rrrrrr] \ar@{^{(}->}[d] & &&&&& M \\
\Omega^{l-1}\,k \oplus \Omega^{l-1}\,k \ar@{..>}[urrrrrr]  \ar@{^{(}->}[d]& &&&&& \\
(\Omega^{l-2}\,k \oplus \Omega^{l-2}\,k)\oplus(\Omega^{l-2}\,k \oplus
\Omega^{l-2}\,k) \ar@{..>}[uurrrrrr] \ar@{^{(}->}[d] \\
\vdots \ar@{^{(}->}[d] \\
(\Omega^1\, k \oplus \Omega^1\,k)\oplus \cdots \oplus(\Omega^1\, k \oplus
\Omega^1\,k) \ar@{..>}[uuuurrrrrr]  \ar@{^{(}->}[dd]\\
\\
 (kV_4 \oplus k \oplus k) \oplus \cdots \oplus (kV_4 \oplus k \oplus k)
\ar@{..>}[uuuuuurrrrrr]}
\]
So it suffices to show  that for a  projective-free $M$ there cannot exist a
factorisation of the form
\[
\xymatrix{
 \Omega^l\, k \ar@{^{(}->}[r] \ar@{^{(}->}[d] & M \\
 (kV_4)^s \oplus k^t \ar@{.>}[ur]_{\phi} & }
\]
where $l$ is a positive integer. It is not hard to see that the invariant
submodule $(\Omega^l\,k)^G$ of $\Omega^l\,k$ maps into $((kV_4)^s)^G$. We will
arrive at a contradiction by showing $((kV_4)^s)^G$ maps to zero under the map
$\phi$. Since $((kV_4)^s)^G \cong ((kV_4)^G)^s$ it is enough to show that
$\phi$ maps each $(kV_4)^G$ to zero. $(kV_4)^G$ is a one-dimensional subspace,
generated by an element which we denote as $v$. If $v$ maps to a non-zero
element, then it is easy to see that the restriction of $\phi$ on the
corresponding copy of $kV_4$ is injective, but $M$ is projective-free, so this
is impossible. In other words $\phi(v) = 0$ and that completes the proof of
the lemma.

\end{proof}

\section{Tate cohomology}

Recall that the Tate cohomology of $G$ with coefficients in a module $M$ is
given by $\Tate^*(G,M) = \uHom_{kG}(\Omega^*k, M)$. In our recent joint work
with Jon Carlson we focused on two fundamental questions (\cite{CarCheMin} and
\cite{CarCheMin2}) about Tate cohomology which we discuss in the next two
subsections. As we shall see, Auslander-Reiten sequences play an important
role in answering both of these questions.

\subsection{Modules with finitely generated Tate cohomology}

Let $M$ be a finitely generated $kG$-module. A well-known result of
Evens-Venkov states that the ordinary cohomology $H^*(G, M)$ is finitely
generated as a module over $H^*(G, k)$. So it is a very natural question to
investigate whether the same is true for Tate cohomology. That is, whether
$\Tate^*(G, M)$ is finitely generated as a module over $\Tate^*(G, k)$. As
shown in \cite{CarCheMin} Tate cohomology is seldom finitely generated. However, there is an
interesting family of modules arising from AR sequences whose Tate cohomology
is finitely generated.

To start, let $N$ be a finitely generated indecomposable non-projective
$kG$-module that is not isomorphic to $\Omega^i k$ for any $i$. Consider
 the Auslander-Reiten sequence
\[ 0 \lrar \Omega^2 N \lrar M \lrar N \lrar 0\]
ending in $N$. Assume that  $N$ has finitely generated Tate cohomology, then
we shall see that the middle term $M$ also has finitely generated Tate
cohomology.

Consider the connecting map $\phi \colon N \lrar \Omega N$ in the exact triangle
\[\Omega^2 N \lrar M \lrar N \lstk{\phi} \Omega N \]
corresponding to the above Auslander-Reiten sequence. We first argue that $\phi$
induces the zero map in Tate cohomology. To this end, let $f \colon \Omega^i k
\rar N$ be an arbitrary map. We want to show that the composite $\Omega^i k
\stk{f} N \stk{\phi} \Omega N$ is zero in the stable category. Now observe
that the assumption on $N$ implies that $f$ is not a split retraction,
therefore the map $f$ factors through the middle term $M$. Since the
composition of any two successive maps in an exact triangle is zero, it
follows that $\phi f = 0$.

Since the boundary map $\phi$ induces the zero map in Tate cohomology, the resulting
long exact sequence in Tate cohomology breaks into short exact sequences:
\[ 0 \lrar \HHHH^*(G, \Omega^2 N) \lrar \HHHH^*(G,  M) \lrar \HHHH^*(G, N) \lrar 0.\]
It is now clear that if $N$ has finitely
generated Tate cohomology, then so does $M$.

In \cite{CarCheMin} we have also shown that the middle term of the AR sequence ending in $k$
has finitely generated Tate cohomology.

These results imply that if we have a sequence of finitely generated
indecomposable non-projective $kG$-modules
\[ N_1,N_2,\cdots,N_s, s \in \mathbb{N} \]
such that $\HHHH^*(G,N_1)$ is finitely generated as a module
over $\HHHH^*(G,k)$ and $N_{i+1}$ is a summand of a middle
term of the Auslander-Reiten sequence which ends in
$N_i, 1 \le i < s$, then all modules $\HHHH^*(G,N_i)$
are finitely generated over $\HHHH^*(G,k)$.

\subsection{Counterexamples to Freyd's generating hypothesis}
Motivated by the celebrated generating hypothesis (GH) of Peter Freyd in homotopy theory \cite{freydGH} and its
analogue in the derived category of a commutative ring \cite{GH-D(R), keir}, we have formulated in \cite{CarCheMin2}
the analogue of Freyd's GH in the stable module category $\stmod(kG)$ of a finite
group $G$, where $k$ is a field of characteristic $p$. This is  the statement that
the Tate cohomology functor detects trivial maps (maps that factor through a projective module) in the thick
subcategory generated by $k$. We studied the GH and related questions in special cases in a series of papers:
\cite{CCM2, CCM3, CCM, CCM4}.  We have finally settled the GH for the stable module category in joint work
with Jon Carlson. The main result of \cite{CarCheMin2} is:

\begin{thm} \cite{CarCheMin2}
The generating hypothesis  holds for $kG$ if and only if the Sylow $p$-subgroup
of $G$ is either $C_2$ or $C_3$.
\end{thm}

Maps in the thick subcategory generated by $k$ that induce the zero map in
Tate cohomology are called ghosts.  Thus non-trivial ghosts give
counterexamples to the GH. After much research on this, we were led to a big
revelation when we discovered that a good source of non-trivial ghosts come
from Auslander-Reiten sequences.

To explain this in more detail, consider an AR sequence
\[ 0 \lrar \Omega^2 N \lrar B \lrar N \lrar 0 \]
ending in $N$. This short exact sequence represents a exact triangle
\[  \Omega^2 N \lrar B \lrar N \lstk{\phi}  \Omega N \]
in the stable category. We will show that
the map $\phi \colon N \lrar \Omega N$
is a non-trivial ghost. First of all, AR sequences are, by definition, non-split short
exact sequences, and therefore the
boundary map $\phi$ in the above triangle
must be a non-trivial map in the stable category.
The next thing to be shown is that the map
$\phi \colon N \lrar \Omega N$ induces the
zero map on the functors
$\uHom_{kG}(\Omega^i k , -)  \cong \hExt^i(k, -)$
for all $i$. Arguing as in section~7.1, consider any map
$f  \colon \Omega^i k  \lrar N$. We have to show that the
composite
\[ \Omega^i k  \lstk{f} N \lstk{\phi} \Omega N \]
is trivial in the stable category. Consider the following diagram
\[
\xymatrix{
 & & \Omega^i k \ar[d]^f \ar@{.>}[dl]  & \\
\Omega^2N \ar[r] & B \ar[r] & N \ar[r]^\phi & \Omega N
}
\]
where the bottom row is our exact triangle. The map $f \colon  \Omega^ik \lrar
N$ cannot be a split epimorphism if we choose $N$ to be an indecomposable
non-projective module in the thick subcategory generated by $k$ that is not
isomorphic to $\Omega^i k$ for any $i$. Then by the defining property of an AR
sequence, the map $f$ factors through the middle term $B$ as shown in the
above diagram. Since the composite of any two successive maps in a exact
triangle is zero, the composite $\phi \circ f$ is also zero by commutativity.
Thus the moral of the story is that in order to disprove the GH for $kG$, we
just have to find an indecomposable non-projective module in the thick
subcategory generated by $k$ that is not isomorphic to $\Omega^i k$ for any
$i$. This was the strategy we used to disprove the GH for $kG$ when the Sylow
$p$-subgroup has order at least $4$.

\section{Galois theoretic connections}

In this section we provide the motivation for some of the
arithmetic objects which will appear in the later sections when we
study Auslander-Reiten sequences in the context of Galois theory.
To set the stage, we begin with our notation.

Let $F$ be a field of characteristic not equal to $2$ and let $F_{sep}$
denote the separable closure of $F$. We shall introduce several subextensions
of $F_{sep}$.
\begin{itemize}
\item $F^{(2)}$ = compositum of all quadratic extensions of $F$.
\item  $F^{(3)}$ = compositum of all quadratic extensions of $F^{(2)}$, which are Galois over $F$.
\item $F^{\{3\}}$ = compositum of all quadratic extensions of $F^{(2)}$.
\item $F_q$ = compositum of all Galois extensions $K/F$ such that $[K: F] = 2^n$, for some positive integer $n$.
\end{itemize}
All of these subextensions are Galois and they fit in a tower
\[ F \subset F^{(2)} \subset F^{(3)} \subset F^{\{3\}} \subset F_q \subset F_{sep} .\]
We denote their Galois groups (over $F$) as
\[G_F \lrar G_q \lrar G_F^{\{3\}} \lrar G_F^{[3]} (=\G_F) \lrar G_F^{[2]} (= E) \lrar 1.\]

Observe that $G_F^{[2]}$ is just $\underset{i \in I}\prod C_2$, where $I$ is
the dimension of $F^*/{F^*}^2$ over $\ftwo$. $F^*$ denotes  the multiplicative
subgroup of $F$. $G_F$ is the absolute Galois group of $F$. The absolute
Galois group of fields are in general rather mysterious objects of great
interest. (See \cite{BLMS} for restrictions on the possible structure of
absolute Galois groups.) Although the quotients $G_q$ are much simpler we are
far from understanding their structure in general. $F^{\{3\}}$ and its Galois
group over $F$ are considerably much simpler and yet they already contain
substantial arithmetic information of the absolute Galois group.

To illustrate this point, consider $WF$ the Witt ring of quadratic forms; see
\cite{Lam} for the definition. Then we have the following theorem.

\begin{thm}\cite{minac-spira-annals}
Let $F$ and $L$ be two fields of characteristic not $2$. Then $WF \cong WL$ (as rings) implies
that $\G_F \cong \G_L$ as pro-$2$-groups. Further if we assume additionally in the case when
each element of $F$ is a sum of two squares that $\sqrt{-1} \in F$ if and only if $\sqrt{-1} \in L$,
then $\G_F \cong \G_L$ implies $WF \cong WL$.
\end{thm}

Thus we see that $\G_F$ essentially controls the Witt ring $WF$ and in fact, $\G_F$ can be viewed as
a Galois theoretic analogue of $WF$. In particular, $\G_F$ detects orderings of fields. (Recall that
$P$ is an ordering of $F$ if $P$ is an additively closed subgroup of index $2$ in  $F^*$ .)
More precisely, we have:

\begin{thm} \cite{minac-spira-mathz}
There is a 1-1 correspondence between the orderings of a field $F$ and cosets
$\{ \sigma \Phi(\G_F) \, |\, \sigma \, \in \, \G_F \backslash \Phi(\G_F)  \text{and }\sigma^2 = 1\}$. Here $\Phi(\G_F)$
is the Frattini subgroup of $\G_F$, which is just the closed subgroup of $\G_F$ generated
by all squares in $\G_F$. The correspondence is as follows:
\[ \sigma \Phi(\G_F) \lrar P_{\sigma} = \{ f \in F^* \, | \, \sigma(\sqrt{f}) = \sqrt{f} \}.\]
\end{thm}

This theorem was generalised considerably for detecting additive properties of multiplicative
subgroups of $F^*$ in \cite{Louis-Minac-Smith}. In this paper (see  \cite[Section 8]{Louis-Minac-Smith})
it was shown that $\G_F$ can be used also for detecting valuations on $F$.

Also in \cite[Corollary 3.9]{adem-dikran-minac} it is shown that $\G_F \cong
\G_L$ if and only if $k_*(F) \cong k_*(L)$. Here $k_*(A)$ denotes the Milnor
$K$-theory (mod 2) of a field $A$. In particular, in  \cite[Theorem
3.14]{adem-dikran-minac} it is shown that if $R$ is the subring of $H^*(\G_F,
\ftwo)$ generated by one dimensional classes, then $R$ is isomorphic to the
Galois cohomology $ H^*(G_F, \ftwo)$ of $F$. Thus we see that $\G_F$ also
controls Galois cohomology and in fact $H^*(\G_F, \ftwo)$ contains some
further substantial information about $F$ which $H^*(G_F, \ftwo)$ does not
contain. In summary, $\G_F$ is a very interesting object. On the one hand
$\G_F$ is much simpler than $G_F$ or $G_q$, yet it contains substantial
information about the arithmetic of $F$.   In \cite[Section 2]{BLMS} the case
of $p > 2$ was also considered, and the definition of $\G_F$ was extended to
fields which contain a primitive $p^{th}$-root of unity. The key reason for
restricting to $p=2$ in earlier papers stems from the interest in quadratic
forms. Nevertheless a number of interesting properties also hold for $\G_F$ in
the case $p > 2$.

\section{Towards the Auslander-Reiten translate $\Omega^2 \ftwo$}

Because $\G_F$ contains considerable information about the arithmetic of $F$, it is a natural question to ask
for properties of $G_q$, or more precisely $G_q(F)$, determined by $\G_F$. First of all, it is possible
that for two fields $F_1$ and $F_2$, we have
\[ \G_{F_1} \cong \G_{F_2} \ \ \text{but }\ \  G_q(F_1) \ncong G_q(F_2).\]
Indeed, set $F_1 = \mathbb{C}((t_1))(( t_2))$ to be a field of formal power series in $t_2$ over the field of formal
power series in $t_1$ over $\mathbb{C}$ and $F_2$ the field $\mathbb{Q}_5$ of $5$-adic numbers.
Then $\G_{F_1} = \G_{F_2} = C_4 \times C_4$, but
$G_q(F_1) = \mathbb{Z}_2 \times \mathbb{Z}_2$ ($\mathbb{Z}_2$ is the additive group of $2$-adic numbers),
while $G_q(F_2) = \la \sigma, \tau \, | \, \tau^4[\tau^{-1}, \sigma^{-1}] = 1 \ra$\
see \cite{minac-spira-annals, Koch}.
However, in this case $G_{F_1}^{\{ 3\}} \cong G_{F_2}^{\{ 3\}}$,  as they are both  isomorphic to
$C_4 \times C_4$; see \cite{AGKM}.

The very interesting question about possible groups $G_F^{\{3\}}$ when
$G_F^{[3]}$ is given, is currently investigated by Min\'{a}\v{c}, Swallow and
Topaz \cite{MST}. In order to tackle this question it is important to
understand the universal case when $G_q(F)$ is a free pro-$2$-group. From now
on we further assume that $F^*/{F^*}^2$ is finite-dimensional.  (This is the
most important case to understand, as $G_F^{\{3\}}$ in general is the
projective limit of its finite quotients.) Let $n$ be the dimension. Then we
have $F^*/{F^*}^2 = 2^n$. In this case we have an extension
\[ 1 \lrar A \lrar G_F^{\{3\}} \lrar E \rar 1\]
where $E = \prod_1^n C_2$ and $A = \Gal(G_F^{\{ 3\}}/F^{(2)}) \cong \prod_1^m C_2$ with $m = (n-1)2^n + 1$.
 Set $R = F^{(2)}$ for simplicity.
Indeed from Kummer theory we know that
\[A \cong \Hom_{\ftwo}(R^*/{R^*}^2, \ftwo).\]
Then also by Kummer theory the minimal number of topological generators of
$G_q^{(2)}: = \Gal(F_q/ F^{(2)})$ is equal to the dimension of $R^*/{R^*}^2$.
From the topological version of Schreier's theorem we know that this number $d(G_q^{(2)})$ is
given by $d(G_q^{(2)}) = (n-1)2^n + 1$, because $G_q^{(2)}$ is the open subgroup of index $2^n$
of $G_q$ and $d(G_q) = n$. (See \cite[Example 6.3]{Koch})
Thus
\begin{eqnarray*}
 \dim_{\;\ftwo} A & = & \dim_{\;\ftwo} R^*/{R^*}^2\\
& = & \dim_{\;\ftwo} H^1(G_q^{(2)}, \ftwo )\\
& = &  d(G_q^{(2)})\\
& = &  (n-1)2^n + 1.
\end{eqnarray*}

Moreover, $A$ is a natural $\ftwo E$-module where the $E$-action is induced by conjugation in
$G_F^{\{3\}}$. The next result gives a completely arithmetical interpretation of the Auslander-Reiten
translate in our case. We write $\Omega^2_E \ftwo$ in order to stress
that we are working in the category of $\ftwo E$-modules.

\begin{thm} \cite{Gasch}
\[ \Omega^2_E \ftwo \cong A.\]
\end{thm}

We sketch in the next section a direct proof of this theorem as it also
provides some information about the structure of our module $A$. After that we
shall consider the AR sequence
\[ 0 \lrar \Omega^2 \ftwo \lrar M \lrar \ftwo \lrar 0\]
and interpret $M$ and also its associated group $AR(E)$ which will be defined later below as
a certain Galois group.

\section{Generators and relations for $\Omega^2 \ftwo$}

Recall that $E = \prod_1^n C_2 = \prod_1^n\la \sigma_i \ra$. We shall first
use the definition of $\Omega^2 \ftwo$ over $\ftwo E$ to describe $\Omega^2
\ftwo$ via generators and relations and then we shall see that indeed $A$ has
this presentation. This, in turn, will imply that  $A \cong \Omega^2 \ftwo$.

Let $P_{a_i}$ be a free $\ftwo E$ module of rank 1 generated by $a_i$ for $1
\le i \le n$, so $P_{a_i} =  \ftwo E a_i$. Then recall that $\Omega^2 \ftwo$
is defined by the short exact sequence:
\[ 0 \lrar \Omega^2 \ftwo \lrar \oplus_{i = 1}^n P_{a_i} (= P) \lstk{\psi} \ftwo E \stk{\epsilon} \ftwo \lrar 0,\]
where $\epsilon \colon \ftwo E \rar \ftwo$ is the augmentation map $ \epsilon
(\sigma_i) = 1$ for all $i$, and $\psi(a_i) = \sigma_i - 1$. We shall denote
$\sigma_i -1$ by $\rho_i$ for simplicity. Exactness at $\ftwo$ is trivial and
exactness at $\ftwo E$ is a standard exercise which we leave to the reader.
Observe that $\Omega^2 \ftwo$ is determined only up to a projective summand.
But since we insist that $\Omega^2 \ftwo$ is projective-free, we actually get
it up to an isomorphism of $\ftwo E$-modules.

Note that $\Omega^2 \ftwo  = \Ker \, \psi$, and it certainly contains the $\ftwo E$-submodule of $P$
generated by
\[S = \{ \rho_i a_i, \rho_j a_i + \rho_i a_j,   1\le  i < j \le n. \}\]
We claim that the $\ftwo E$ submodule $W$ generated by $S$ is the entire
module $\Omega^2 \ftwo$. First observe that $I = \Omega^1 \ftwo = \Ker \,
\epsilon$ has dimension $2^n -1$ and
 $\rho_T = \prod_{i \in T } \rho_i$, where $T$ is a non-empty subset of $\{1, 2, \cdots, n\}$
form a basis of $I$. This allows us to pick a nice $\ftwo$-vector space section of $\psi$:
$V$ the vector space span of the set
\[ U = \{ \rho_T a_i \ \ 1 \le i \le n, \text{and } T \subseteq \{i +1, i+2 , \cdots, n \} \}.\]
 It is not hard to see
that $P$ splits as $\Ker \, \psi \oplus V$ as $\ftwo$-vector spaces. Therefore
it is enough to show that $W \subseteq \ker \, \psi$ together with $V$
generate the  entire module $P$ over $\ftwo$. Set $Q := W \oplus V$, the
vector span of $V$ and $W$ over $\ftwo$. Since $\{a_1, a_2, \cdots, a_n\}
\subseteq Q$ and $\{a_1, a_2, \cdots, a_n \}$  generates $P$ as an $\ftwo
E$-module, it is sufficient to show that $Q$ is an $\ftwo E$ submodule of $P$.
Because $W$ is an $\ftwo E$-submodule of $Q$ it is enough to show that for
each $v \in U$ and $\rho_l$ we also have $\alpha := \rho_l v$ in $Q$. Write $v
= \rho_T a_i$ for some $T$ and $i$. If $l = i$, then $\alpha$ belongs to $W$.
If $l > i$, then $\alpha$ belongs to $U \cup \{0 \}$. In particular, in both
cases, $\alpha \in Q$. If $l < i$, then
\begin{eqnarray*}
\alpha & = &  \rho_T \rho_l a_i \\
       & = & \rho_T (\rho_l a_i + \rho_i a_l) + \rho_T \rho_i a_l \\
       & = & w + v
\end{eqnarray*}
where the first summand belongs to $W$ and the second summand belongs to $U \cup \{ 0 \} \subseteq V$.
Hence $\alpha$ belongs to $Q$, as desired. This shows that $\Omega^2 \ftwo$ is generated over $\ftwo E$ by  the set $S$.

Let us denote:
\[ [a_i, a_j, a_{t_1}, \cdots , a_{t_r}] : = \rho_{t_r}\rho_{t_{r-1}} \cdots \rho_{t_1} (\rho_i a_j + \rho_j a_i),  \
 \text{and}\]
 \[ [{a_i}^2,  a_{t_1}, \cdots , a_{t_r}] : = \rho_{t_r}\rho_{t_{r-1}} \cdots \rho_{t_1} \rho_i a_i, \]
where all the subscripts range over the set $\{1, 2, \cdots, n\}$. Then we
have the following identities. Below $\sigma$ denotes an arbitrary permutation
of the set $\{1, 2, 3, \cdots, s\}$, and $b_i = a_{t_i}$ for any $i$.
\begin{enumerate}
\item $[a_i, a_j] = [a_j, a_i]$
\item $[a_i, a_j, b_1, \cdots, b_s] = [a_i, a_j, b_{\sigma(1)}, \cdots , b_{\sigma(s)}]$
\item $[a_i^2, b_1, \cdots , b_s] = [a_i^2, b_{\sigma(1)}, \cdots,
b_{\sigma(s)}]$
\item $[a_i, a_j, a_r] + [a_j,  a_r,  a_i] + [a_r,  a_i, a_j] = 0$
\item $[a_i^2, a_r] = [a_i, a_r, a_i]$
\item $[a_i, a_j, b_1, \cdots, b_{\alpha}, \cdots, b_{\alpha}, \cdots, b_s] = 0$
\end{enumerate}

Then $\Omega^2 \ftwo \subseteq P$ is a vector span over $\ftwo$ of $[a_i, a_j,
a_{t_1}, \dots a_{t_r}]$ and $[a_i^2, a_{t_1}, \cdots a_{t_r}]$ as above. In
fact from the short exact sequence
\[ 0 \lrar \Omega^2 \ftwo \lrar P \lstk{\psi} \ftwo E \stk{\epsilon} \ftwo \lrar 0,\]
we see that $\dim_{\ftwo} \Omega^2 \ftwo = n2^n - 2^n + 1 = 2^n(n-1) + 1$.
With what follows using the above identities we see that the set
\[ X := \{ [c_{ij}, a_{t_1}, \cdots, a_{t_r}]  \, |\, 1\le i \le j \le n, i < t_1 < \cdots < t_r \le n\}\]
generates $\Omega^2 \ftwo$ over $\ftwo$. Here $c_{ij} = [a_i ,a_j]$ if $i < j$ or $a_i^2$
if $i = j$. Observe that
\[|X| = \sum_{i = 1}^n (n-i+1) 2^{n-i}. \]
Indeed for a fixed $i$ there are $n-i+1$ possibilities for $j$ and there are
$2^{n-i}$ possibilities for the number of subsets of $\{i+1, \cdots, n \}$ for
the choices of $t_1, \cdots, t_r$. Also we have
\[ \sum_{i=1}^n (n-i+1) Y^{n-i} = \sum_{i=1}^n (Y^{n-i+1})' = \frac{(n+1)Y^n(Y-1)-(Y^{n+1}-1)}{(Y-1)^2}.\]
Plugging $Y = 2$, we see that $|X| = \sum_{i=1}^n (n-i+1)2^{n-i} = 2^n (n-1) +
1$. Hence $X$ is the basis of $\Omega^2 \ftwo$ over $\ftwo$. On the other
hand, a detailed investigation of identities in $G_F^{[3]}$ (see \cite{GuTa, MST}) shows
that a basis of $A$ over $\ftwo$ is
\[ Z = \{ [d_{ij}, \sigma_{t_1}, \cdots \sigma_{t_r}] \, | \,1 \le i \le j \le n, i < t_1 < \cdots < t_r \le n \},\]
$d_{ij} = [\sigma_i, \sigma_j]$ if $i < j$, $d_{ii} = \sigma_i^2$, and
$[\sigma_1, \sigma_2, \cdots, \sigma_l] = [\cdots[[[\sigma_1, \sigma_2],
\sigma_3], \cdots, \sigma_l]$. In fact, writing $\sigma_i$ in place of $a_i$
above, the same set of 6 identities hold in $G_F^{[3]}$. This shows that the
map $\phi \colon \Omega^2 \ftwo \lrar A$ defined by $\phi(a_i^2) =
\sigma_i^2$, and $\phi([a_i, a_j]) = [\sigma_i, \sigma_j]$ is a well-defined
surjective $\ftwo E $ homomorphism and because of dimensional reasons it has to
be an isomorphism.

\section{The Auslander-Reiten sequence ending in $\ftwo$}

Finally we want to obtain the Auslander-Reiten sequence ending in $\ftwo$:
\[ 0 \lrar \Omega^2 \ftwo \lrar M \lstk{\eta} \ftwo \lrar 0.\]
We will get a sequence of this form which is not split. Then simply because
\[\Ext_{\ftwo E}^1(\ftwo, \Omega^2 \ftwo) = \ftwo\]
the non-split sequence we obtain has to be an Auslander-Reiten sequence; see
\cite[Proposition (78.28)]{Curtis-Reiner}.

Observe from our proof of our description of $\Omega^2 \ftwo \subseteq P$,
that $a:= \rho_n \rho_{n-1} \cdots \rho_2 a_1$ belongs to $P - \Omega^2
\ftwo$.  Let $M := \Omega^2 \ftwo + \{a, 0\} \subseteq P$. (Because $\rho_i a$
belongs to $\Omega^2 \ftwo$ for each $i = 1, 2, \cdots, n$, we see that $M$ is
indeed an $\ftwo E$-module.) Now,  let $\eta \colon M \rar \ftwo$ be defined
by $\eta (\Omega^2 \ftwo) = 0 $ and $\eta (a) = 1$. Observe that since $\rho_1
a$ belongs to $\Omega^2 \ftwo$ and $\rho_i a = 0$ for $i \ge 2$, we obtain
that $\eta$ is indeed an $\ftwo E$-homomorphism.

We claim that the short exact sequence
\[ 0 \lrar \Omega^2 \ftwo \lrar M \lstk{\eta} \ftwo \lrar 0\]
is an Auslander-Reiten sequence. Suppose to the contrary, this sequence
splits. Then the coset $a + \Omega^2 \ftwo$ would contain a trivial element $t
= a + m$, $m \in \Omega^2 \ftwo$ (i.e., the action of each $\rho_i$ on $t$ is
$0$). But $t \in \text{Soc} (P) \subseteq \Omega^2 \ftwo$. Hence $t - m = a$
belongs to $\Omega^2 \ftwo$, which is a contradiction. (Here $\text{Soc} (P)$
is the socle of $P$ which is the submodule of $E$ consisting of the fixed elements
of $P$ under the action of $E$.) So we are done.

We now introduce Auslander-Reiten groups. Let $E = \prod_{i=1}^n C_2 =
\prod_{i=1}^n \la \sigma_i\ra$ as before. Assume that $n \ge 2$. Consider the
Auslander-Reiten sequence
\[ 0 \lrar \Omega^2 \ftwo \lstk{h} M \lrar \ftwo \lrar 0. \]
Consider the long exact sequence in cohomology associated to the above AR
sequence.
As observed before, all connecting
homomorphisms are zero except the first one. In particular, we have a short
exact sequence:
\[ 0 \lrar H^2(E, \Omega^2 \ftwo) \lrar H^2(E, M) \lrar H^2(E, \ftwo) \lrar 0. \]
Furthermore, $H^2(E , \Omega^2 \ftwo) \cong \ftwo$. The unique non-zero
element of $H^2(E, \Omega^2 \ftwo)$ corresponds to a short exact sequence
\[1 \lrar \Omega^2 \ftwo \lrar G_F^{\{3\}} \lrar E \lrar 1\]
where $|F^*/{F^*}^2| = 2^n$ and $\Gal(F_q/F)$ is a free pro-$2$-group. Note
that in our case $G_F^{\{3\}}$ does not depend on $F$, but only on the number $n$
of the rank of the free absolute pro-$2$-group $G_F$. Therefore we shall
denote it by $G^{\{3\}}(n)$. Now consider the following commutative diagram
\[
\xymatrix{
1 \ar[r] & \Omega^2 \ftwo \ar[r] \ar[d]^h & G^{\{3\}}(n) \ar[r] \ar[d] & E \ar[r] \ar[d]^{=} & 1 \\
1 \ar[r] & M \ar[r] & T(n) \ar[r]  & E \ar[r] & 1 }
\]
Here $T(n)$ is the pushout of $G^{\{3\}}(n)$ and $M$ along $\Omega^2 \ftwo$. The
second row corresponds to the image of the non-trivial element of $H^2(E,
\Omega^2 \ftwo)$ in $H^2(E, M)$.

\begin{defn}
We call $T(n)$ the Auslander-Reiten group associated to $E$.
 \end{defn}

Observe that we have obtained a Galois theoretic model for our
Auslander-Reiten sequence:
\[ 0 \lrar \Omega^2 \ftwo \lrar M \lrar \ftwo \lrar 0. \]
Indeed, let $K/L$ be a Galois extension as above such that $\Gal(K/F) = T(n)$.
Then we have a short exact sequence
\[ 0 \lrar M \lrar T(n) \lrar G \lrar 0\]
where AR Galois group $T(n)$ is generated by $G^{\{3\}}(n)$ and $\tau$ subject
to the relations $\tau^2 = 1$,  $[\tau, \sigma_i] = 1$ for $2 \le i \le n$,
and $[\tau, \sigma_1] = [\sigma_1^2 , \sigma_2, \cdots, \sigma_n]$. Let $L$ be
the fixed field of $M$. Then the Galois group $\Gal(K/L) \cong M$ is the
middle term of our Auslander-Reiten sequence. Further, let $H$ be a subgroup
generated by $\sigma_1, \cdots, \sigma_n$. Then $H \cong G^{\{3\}}(n)$ and $H \cap M \cong \Omega^
2\ftwo$. Hence our Auslander-Reiten sequence has the form
\[ 0 \lrar H \cap \Gal(K/L) \lrar \Gal(K/L) \lrar \ftwo \lrar 0.\]

\section{Conclusion}
We hope that our hot and mild appetizers will inspire our readers in their choices of
further delightful, tasty, and mouth-watering main dishes.  John Labute's skillful
interplay between Lie theoretic and group theoretic methods in his past work (see
for example \cite{Labute-67}), as well as in his recent work (\cite{Labute-06}) provides us with
inspiration for exploiting the connections between modular representation theory
and Galois theory. We are now lifting our glasses of champagne in honour of
John Labute and we are wishing him happy further research, and pleasant
further Auslander-Reiten encounters.


\end{document}